\documentclass[12pt,english]{amsart}
\usepackage[T1]{fontenc}
\usepackage[latin9]{inputenc}
\usepackage{amsthm}
\usepackage{amstext}
\usepackage{amssymb}
\usepackage{esint}

\usepackage{color}

\makeatletter
\numberwithin{equation}{section}
\numberwithin{figure}{section}
\theoremstyle{plain}
\newtheorem{thm}{\protect\theoremname}
\theoremstyle{plain}
\newtheorem{prop}[thm]{\protect\propositionname}
\theoremstyle{plain}
\newtheorem{rem}[thm]{\protect\remarkname}

\@ifundefined{date}{}{\date{}}

  \newcommand{\curl}{\nabla\times}
  \newcommand{\eps}{\varepsilon}
  \newcommand{\ep}{\epsilon}
  \renewcommand{\div}{\text{div}}
  \newcommand{\R}{\mathbb R}
  \newcommand{\C}{\mathbb C}

  \newcommand{\supp}{\text{supp}}

\makeatother

\usepackage{babel}
  \providecommand{\propositionname}{Proposition}
\providecommand{\theoremname}{Theorem}
\providecommand{\remarkname}{Remark}

\begin{document}

\title[Size estimate]{Uniqueness estimates for the general complex conductivity equation and their applications to inverse problems}

\author[Carstea]{C\u at\u alin I. C\^arstea}
\address{ HKUST Jockey Club Institute for Advanced Study, The Hong Kong University of Science and Technology, Clear Water Bay, Kowloon, Hong Kong}
\curraddr{}
\email{catalin.carstea@gmail.com}

\author[Nguyen]{Tu Nguyen}
\address{Institute of Mathematics, Vietnam Academy of Science and Technology,
18 Hoang Quoc Viet, Cau Giay, Hanoi, Vietnam}
\curraddr{}
\email{natu@math.ac.vn}
\thanks{The second author is funded by Vietnam National Foundation for Science and Technology Development (NAFOSTED) under grant number 101.02-2015.21}

\author[Wang]{Jenn-Nan Wang}
\address{Institute of Applied Mathematical Sciences, NCTS, National Taiwan University, Taipei 106, Taiwan}
\curraddr{}
\email{jnwang@math.ntu.edu.tw}
\thanks{The third author was supported in part by MOST 105-2115-M-002-014-MY3}


\begin{abstract}
The aim of the paper is twofold. Firstly, we would like to derive quantitative uniqueness estimates for solutions of the general complex conductivity equation. It is still unknown whether the \emph{strong} unique continuation property holds for such equations. Nonetheless, in this paper, we show that the unique continuation property in the form of three-ball inequalities is satisfied for the complex conductivity equation under only Lipschitz assumption on the leading coefficients. The derivation of such estimates relies on a delicate Carleman estimate. Secondly, we study the problem of estimating the size of an inclusion embedded inside of a conductive body with anisotropic complex admittivity by one boundary measurement. The study of such inverse problem is motivated by practical problems.
\end{abstract}

\dedicatory{Dedicated to Professor Carlos Kenig on the occasion of his 65th birthday}

\maketitle

\section{Introduction}

The study in this paper is inspired by a problem arising in Electric impedance tomography (EIT).  EIT is a non-invasive method for inferring the interior structure of a specimen by injecting currents and measuring the resulting potential differences on the surface of the specimen. Implementing EIT as a medical imaging technique was first introduced by Henderson and Webster \cite{hw78}. The mathematical study of EIT was initiated by Calderon \cite{ca80}.  There have been a lot of clinical and non-clinical applications of EIT, see \cite{cin} and references therein for more detailed descriptions.

Ideally, one would like to determine the full information of the interior structure by making boundary measurements. However, according to the uniqueness result of Sylvester and Uhlmann  \cite{su87}, we need infinitely many measurements, which is practically impossible. In some cases of medical diagnosis, one is more interested in determining the characteristic properties of an abnormality surrounded by normal or healthy tissues, for example,  breast cancer detection \cite{cin} and detectability of the degraded region \cite{gr}.  It turns out EIT with one measurement can provide a preliminary assessment of the size of the abnormality \cite{ars}. 

In \cite{ars}, the authors considered a real conductivity equation. However, to study inverse problems related to the biological tissues, it is more appropriate to use a complex conductivity equation. Unlike metallic materials, the electrical conductance within biological tissues is due to ions which exist in intracellular and extracellular fluids.  The intracellular ions are separated by cell membranes which act like capacitors. When a constant voltage (DC) is applied, the currents flow only through extracellular fluids. We characterize such currents as conductance currents $I_c$. At the same time, a constant amount of charge will be stored in the cell membrane capacitors. If now an alternating voltage (AC) is applied, the charges stored in the cell membranes are changing in terms of the frequency, and this results in a flow of current through intracellular fluids. We characterize these currents as displacement currents $I_d$. The total current $I$ is the sum of the conductance and the displacement currents. Note that the displacement current is 90 degrees apart in phase. From Ohms' law, the conductance current $I_c$ depends on the material's conductivity $\sigma$ and the displacement current $I_d$ depends on its permittivity $\eps$. Due to the phase difference, we can see that the total current $I$ is proportional to the product of the complex-valued admittivity $\sigma+i\omega\eps$ and the applied voltage, where $\omega=2\pi f$ and $f$ is the frequency of the applied signal \cite{mph06}. 

Besides  considering these electrical properties of biological tissues, we would like to point out that biological tissues are very inhomogeneous. The cells of biological tissues have different sizes and consist of different compositions. Furthermore, some biological tissues are extremely anisotropic such as bones and muscles. In this paper, the material's inhomogeneity and anisotropy are also taken into account. Finally, for most biological tissues,  $\sigma$ and $\eps$ also depend on the frequency $\omega$. Since we only consider one pair of current-voltage measurement, we will simply choose $\omega=1$.    

Now we are ready to formulate our inverse problem. Let $\Omega$ be a conducting body with an anisotropic background with current-voltage relation is given by
 \[\gamma_0(\nabla u)=(\sigma_0(x)+i\epsilon_0(x))\nabla u(x), \]
where the real symmetric matrices $\sigma_0$ and $\epsilon_0$ represent the electrical conductivity and the electrical permittivity, respectively. Assume that $D\subset \Omega$ represents an inclusion strictly embedded in $\Omega$ in which the medium in $D$ is different from the background one $\gamma_0$. Here we will consider a material in $D$ where the current-voltage in $\Omega$ is described by the following constitutive law: 
\begin{equation}\label{exotic}
\gamma_1(\nabla u)=(\sigma_1(x)+i\epsilon_1(x))\nabla u+\zeta_1 \overline{\nabla u},
\end{equation}
where $\sigma_1,\eps_1$, $\zeta_1$ are real symmetric matrix-valued functions. For convenience, we will denote $\zeta_0 \equiv 0$ and assume that the supports of $\sigma_1-\sigma_0$, $\eps_1-\eps_0$, $\zeta_1-\zeta_0$ are contained in $D$. 
The perturbed equation is 
\begin{equation*}
\nabla\cdot(\gamma_1(\nabla u))=0\quad\mbox{in}\quad\Omega.
\end{equation*}

Before proceeding to the description of the inverse problem, we would like to provide a real physical model that gives rise to the relation \eqref{exotic}. In addition to the biological tissues, complex conductivity equations also arise in modeling the propagation of electromagnetic waves. We consider the electromagnetic waves propagating in a chiral medium. We recall the Drude-Born-Fedorov constitutive relations (see for instance \cite[(2.21)]{RSY}):
\begin{equation}\label{de1}
{\bf D}=\eps({\bf E} +\beta\curl{\bf E}),\quad {\bf B}=\mu ({\bf H} + \beta\curl {\bf H}),
\end{equation}
where ${\bf E}, {\bf H}$ are the electric and magnetic fields, and ${\bf D}, {\bf B}$ are the electric and magnetic flux densities, respectively. Also, in the relations, $\eps$ is the electric permittivity, $\mu$ is the magnetic permeability, and $\beta$ is the chirality. The source-free time-harmonic equations are described by
\begin{equation}\label{de2}
\curl{\bf  E}-i\omega{\bf B}=0,\quad\curl{\bf H}+i\omega{\bf D}=0.
\end{equation}
Substituting \eqref{de1} into \eqref{de2} leads to
\begin{equation}\label{de22}
\curl{\bf E}-i\omega\mu({\bf H} + \beta\curl {\bf H})=0,\quad \curl{\bf H}+i\omega\eps({\bf E} +\beta\curl{\bf E})=0,
\end{equation}
which can be combined to obtain
\begin{equation}\label{de3}
\curl{\bf E}=\eta\beta{\bf E}+i\omega\frac{\mu}{(1-k^2\beta^2)}{\bf H},\quad\curl{\bf H}=\eta\beta{\bf H}-i\omega\frac{\eps}{(1-k^2\beta^2)}{\bf E},
\end{equation}
where $\eta=k^2(1-k^2\beta^2)^{-1}$ and $k^2=\omega^2\eps\mu$. Here we consider the case where $\beta$ is real and $\eps$, $\mu$ are complex-valued satisfying that $\eps\mu$ is real and $|\mu|$ is small. Possible choices of $\eps$ and $\mu$ are $\eps=\rho e^{-i\theta}$ and $\mu=\tilde\rho e^{i\theta}$, where $\theta\in(0,2\pi)$, $\rho>0$, $\tilde\rho>0$ with $|\tilde\rho|\ll 1$. For our approximation purpose later, we also assume that $\theta$ is small so that we can ignore the imaginary part of $\xi=\sqrt{\mu/\eps}$. To justify our assumptions, we consider the parameters at different scales. Let $\delta>0$ be sufficiently small, $\rho=O(1), \tilde\rho=O(\delta)$, and $\theta=O(\delta)$. Then we can easily see that $\eps\mu=O(\delta)$ and $\xi=O(\delta^{1/2})+iO(\delta^{3/2})$. So it is not unreasonable to ignore the imaginary part of $\xi$. Since we assume $|\mu|\ll 1$, we may consider $k^2\beta^2<1$. 

In view of the Bohren decomposition \cite[p.87]{RSY}, we write ${\bf E}$ and ${\bf H}$ in terms of the Betrami fields ${\bf Q}_L$, ${\bf Q}_R$, as  
\begin{equation}\label{ql}
{\bf E}={\bf Q}_L-i\xi{\bf Q}_R,\quad{\bf H}=\frac{1}{i\xi}{\bf Q}_L+{\bf Q}_R,
\end{equation}
where $\xi=\sqrt{\mu/\eps}$ is the intrinsic impedance of the chiral medium and ${\bf Q}_L$, ${\bf Q}_R$ are Beltrami fields satisfying
\begin{align*}
\begin{split}
\curl {\bf Q}_L&=\gamma_L{\bf Q}_L,\quad\gamma_L=k(1-k\beta)^{-1},\\
\curl {\bf Q}_R&=-\gamma_R{\bf Q}_R,\quad\gamma_G=k(1+k\beta)^{-1}.
\end{split}
\end{align*}
Here $\gamma_L$ and $\gamma_R$ are the wave numbers of the Beltrami fields ${\bf Q}_L$ and ${\bf Q}_R$, respectively. Since $\gamma_L$ and $\gamma_R$ are assumed to be real, it suffices to choose real ${\bf Q}_L$, ${\bf Q}_R$. Reasoning as above, the imaginary part of $\xi$ is negligible, $\xi$ could be considered as real-valued.  Consequently, we can approximately set
\[
{\bf H}=-\frac{i}{\xi}\overline{\bf E}.
\]
Substituting this into the second equation of \eqref{de3} implies
\[
\curl{\bf H}=-i\frac{\eta\beta}{\xi}\overline{\bf E}-\frac{i\omega\eps}{(1-k^2\beta^2)}{\bf E},
\]
which gives
\begin{equation}\label{de5}
\nabla\cdot\left(\frac{\eta\beta}{\xi}\overline{\bf E}+\frac{\omega\eps}{(1-k^2\beta^2)}{\bf E}\right)=0.
\end{equation}
Taking into account of the fact that $|\mu|$ is small, from the first equation of \eqref{de22}, we set $\curl{\bf E}=0$. In other words, we can find a potential function $u$ such that ${\bf E}=\nabla u$. Replacing ${\bf E}$ by $\nabla u$ in \eqref{de5} gives the current-voltage relation \eqref{exotic}.

Given a Neumann boundary data $h\in L^2(\partial\Omega)$ satisfying $\int_{\partial\Omega}h=0$,  let $u_0$ be the unique solution of the unperturbed equation 
\begin{equation}\label{unpert}
\left\{
\begin{aligned}
&\nabla\cdot(\gamma_0\nabla u_0)=0\quad\text{in}\quad\Omega,\\
&\gamma_0\nabla u_0\cdot\nu=h\quad\text{on}\quad\partial\Omega,\\
&\int_{\Omega}u_0=0,
\end{aligned}
\right.
\end{equation}
and $u$ be the solution of the perturbed equation with the same boundary data, i.e.
\begin{equation}\label{pert}
\left\{
\begin{aligned}
&\nabla\cdot(\gamma_1(\nabla u))=0\quad\text{in}\quad\Omega,\\
&\gamma_1(\nabla u)\cdot\nu=\gamma_0\nabla u\cdot\nu=h\quad\text{on}\quad\partial\Omega,\\
&\int_{\Omega}u=0.
\end{aligned}
\right.
\end{equation}
The inverse problem studied here is to estimate the size of the inclusion $D$ by one pair of boundary measurement $\{u|_{\partial\Omega},h\}$. More precisely, one tries to estimate  $|D|$ using the power
gap $\delta W=W_0-W$, where
\[
W=\int_{\partial\Omega}u h,\quad W_0=\int_{\partial\Omega}u_0 h.
\]
Previously, the size estimate problem for the conductivity equation with complex coefficients was studied by Beretta, Francini, and Vessella \cite{bfv} where they assumed  either $\gamma_0$, $\gamma_1=\sigma_1+i\ep_1$ (with $\zeta_1=0$) are complex constants or $\gamma_0$ is real valued.

Similar to the arguments used in \cite{ars}, our method relies on some quantitative uniqueness estimates of the solution $u_0$. Therefore, in the first part of the paper we study local behaviors of solutions to the following inequality
\begin{equation}\label{de}
|Pu|\le K(|\nabla u|+|u|),
\end{equation}
 where  $Pu=\sum_{jk}p_{jk}(x)\partial_{jk}u$ is a  second order elliptic equation with complex, symmetric Lipschitz coefficients $p_{jk}(x)=a_{jk}(x)+ib_{jk}(x)$. We assume the following ellipticity condition
\begin{equation}\label{ellip}
\lambda\left|\xi\right|^{2}\leq\sum_{jk}a_{jk}(x)\xi_{j}\bar{\xi}_{k}\leq\lambda^{-1}\left|\xi\right|^{2}, \quad |p_{jk}(x)|\leq \lambda^{-1},\quad\forall\ x\in\R^{n},\ \xi\in\C^{n}
\end{equation}
for some $\lambda<1$. 

When $P$ has real coefficients, i.e. $b_{jk}\equiv 0$,  qualitative and quantitative forms of strong unique continuation for \eqref{de} are well-known. These results cease to hold in the general case of complex coefficients, as Alinhac \cite{al} has constructed such operators of any order in $\R^2$ without the strong unique continuation property. The principal symbols of these operators have two simple, non-conjugate complex characteristics.  However,  it was proved in \cite{cgt} that the strong unique continuation property holds when $P$ has Gevrey coefficients.

One of the aims of the paper is to prove the following  three-ball inequality for functions satisfying \eqref{de}.
\begin{thm}\label{T1}
Assume that \eqref{ellip} holds, $\|\nabla p\|_{L^{\infty}(\R^{n})}\le L$. Let  $R=\min\{1,\frac{\lambda}{10nL}\}$.  Then there exists positive constant  $s=s(\lambda,K)$ such that if  $u\in H_{loc}^1(\Omega)$ solves \eqref{de} and $ r_{0} <r_{1}< \lambda r_{2}/2 <\sqrt\lambda R/2$ with $B_{r_{2}}(x_0)\subset\Omega$, we have
\begin{equation}\label{3ball}
\int_{B_{r_{1}}(x_0)}|u|^2dx\leq C\left(\int_{B_{r_0}(x_0)}|u|^2dx\right)^{\tau}\left(\int_{B_{r_{2}}(x_0)}|u|^2dx\right)^{1-\tau},
\end{equation}
where $C$ depends on $\lambda,K,r_1/r_0,r_2/r_0,s$ and
\[
\tau=\frac{(2r_1/\lambda)^{-s}-r_2^{-s}}{r_0^{-s}-r_2^{-s}}.
\]
\end{thm}

In the second part of the paper, we will investigate the size estimate problem. In this work, we consider the size estimate problem with anisotropic complex-valued $\gamma_0$ and suitable medium $\gamma_1$ in the inclusion $D$. We will study the physically relevant model where the real part of $\gamma_0$ and $\sigma_1, \zeta_1$ are positive-definite. One of the main tasks is to bound the power gap by the "energy" of the unperturbed solution inside the inclusion $D$.   The derivation relies on a key idea. That is, we can express the conductivity equation with complex coefficients in a convex variational form with real coefficients introduced by Cherkaev and Gibiansky \cite{cg94}. However, this technique alone is not sufficient, as the difference of the perturbed and the unperturbed matrices (i.e., the matrices which appear in the constitutive relations) is always indefinite no matter what jump relations one imposes on material parameters. Physically, this indefiniteness is attributed to the dissipation of electric energy. Such indefiniteness makes it impossible to bound the power gap below by the energy of the unperturbed solution in $D$, which is always positive. Considering special material satisfying the current-voltage \eqref{exotic} allows us to overcome the indefiniteness and derive needed energy inequalities under suitable jump condition across the boundary $\partial D$. We would like to point out that the derivation of  energy inequalities in \cite{bfv} only works under their assumption that either $\sigma_0+i\ep_0$, $\sigma_1+i\ep_1$ (with $\zeta_1=0$) are complex constants or $\gamma_0$ is real-valued.

We note that when the admittivities of the inclusion and the surrounding body are known, some bounds for $|D|$ were derived by Thaler, Milton \cite{tm} using the splitting method, and by Kang, Kim, Lee, Li, Milton \cite{kkllm} using the translation method. 

This paper is organized as follows. In Section~2, we will derive the Carleman estimate and three-ball inequalities. Section~3 is devoted to the study of the size estimate problem.

\section{Carleman estimate and consequences}

In this section we will derive the needed three ball inequality. First, we recall the following Carleman estimate of \cite{cgt} (or \cite{coko}). Let $r=\left|x\right|$ and denote $\varphi(x)=r^{-s}$. 
\begin{prop}\label{car}
Let $P=\sum_{jk}p_{jk}(x)\partial_{jk}$ be a second order elliptic operator with symmetric Lipschitz coefficients
$p_{jk}=a_{jk}+ib_{jk}$ in which \eqref{ellip} holds. Assume also that $a_{ij}(0)=\delta_{ij}$. 
Suppose that $\|\nabla p\|_{L^{\infty}(\R^{n})}\le L<\infty$. Then there exist positive constants $s_{0}$,$\beta_{0}$ and $C$
depending only on $n$ and $\lambda$, and $R$ depending on $n$, $\lambda$, $L$, so that if $\beta\geq\beta_{0}$,
$s\geq s_{0}$, for any $u\in H_{0}^{1}(B_{R}\backslash\{0\})$
the following inequality holds
\begin{equation}\label{eq4}
\beta^{3}s^{3}\int r^{-3s-4}e^{2\beta\varphi}|u|^{2}dx+\beta s^{}\int r^{-s-2}e^{2\beta\varphi}\left|\nabla u\right|^{2}dx\leq C\int e^{2\beta\varphi}\left|Pu\right|^{2}dx.
\end{equation}
\end{prop}
\begin{rem}
Paper \cite{cgt} mainly concerns the strong unique continuation property for the operator $P$ with Gevrey coefficients $p_{jk}(x)$. There, the upper bound of the exponent $s$ in \eqref{eq4} is controlled by the Gevrey degree and its lower is determined by the eigenvalues of $a_{jk}(0)$ and $b_{jk}(0)$. Here we are interested in the quantitative estimate of the unique continuation property. Thus, $s$ can be chosen arbitrarily large. We only need to make sure that the lower bound $s_0$ is a general constant, which can be easily verified from the proof of \eqref{eq4} in \cite{cgt}. 

Carleman estimates for general elliptic equations with complex coefficients were also derived in \cite{br15} and \cite{br16} using H\"ormander's approach. However, it does not seem possible to derive three-ball inequalities like \eqref{3ball} from their estimates.
\end{rem}

Having the Carleman estimate \eqref{eq4}, we now prove the three-ball inequality in Theorem~\ref{T1}.
\begin{proof}[Proof of Theorem \ref{T1}] 
Without loss of generality, we assume $x_0=0$. To apply the Carleman estimate, we first need to make a change of variables. Let $S$ be a matrix satisfies $SA(0)S^t=I$ and let $v(y)=u(S^{-1}y)$, $\tilde{P}(y)=SP(S^{-1}y)S^t$. Then $v$ satisfies $\tilde{p}_{ij}\partial_{y_j y_k}v=0$. Note that we have $\Re \tilde{P}(0)=I$.

 Let $\rho_0 = \sqrt\lambda r_0, \rho_1 = r_1/\sqrt\lambda $, and $\rho_2 = \sqrt\lambda r_2$. Let $0\le \chi \le 1$ be a cut-off function satisfying 
\begin{itemize}
\item $\chi(y)=1$ if $2\rho_{0}/3<|y|< \rho_{2}/2$,
\item $\chi(y)=0$ if $|y|\le \rho_0/2$ or $|y|\ge 2\rho_2/3$, and,
\item $\left| \partial^{\alpha}\chi(y)\right|\le 10|y|^{-\left|\alpha\right|},\, \forall y$, for $\left|\alpha\right|=1,2$.
\end{itemize}
Note that we have $0<\rho_0<\rho_1<\rho_2/2\le R/2$, where $R$ is the constant appears in Proposition~\ref{car}.  Denote $\tilde v=\chi v$. Applying \eqref{eq4} to $\tilde v$ and using \eqref{de}, we have that
\begin{align}\label{eq5}
&\beta^{3}s^{3}\int |y|^{-3s-4}e^{2\beta\varphi}|\tilde v|^{2}dx
+\beta s^{}\int |y|^{-s-2}e^{2\beta\varphi}|\nabla \tilde v|^{2}dx \notag\\
&\leq K^2\int e^{2\beta\varphi}\left(|\tilde v|^{2}+|\nabla\tilde v|^2\right)dx\notag \\
&+C_{\lambda,K}\int_{\{\rho_0/2<|y|<2\rho_0/3\}\cup\{\rho_2/2<|y|<2\rho_2/3\}} |y|^{-4}e^{2\beta\varphi}\left(|v|^{2}+|\nabla v|^2\right)dx,
\end{align}
where $C_{\lambda,K}$ is a constant depending on $\lambda$ and $K$.  Taking $s>K$ , then the first term of the right hand side of \eqref{eq5} is absorbed by its left hand side. Consequently, we obtain
\begin{align}\label{eq6}
&\int_{2\rho_0/3<|y|<\rho_1} |y|^{-3s-4}e^{2\beta\varphi}|v|^{2}dy\notag\\
&\le C_{\lambda,K}\int_{\{\rho_0/2<|y|<2\rho_0/3\}\cup\{\rho_2/2<|y|<2\rho_2/3\}} e^{2\beta\varphi}\left(|y|^{-4} |v|^{2}+|y|^{-2}|\nabla v|^2\right)dy.
\end{align}
It follows that
\begin{align}\label{eq66}
&e^{2\beta \rho_1^{-s}}\rho_1^{-4}\int_{2\rho_0/3<|y|<\rho_1}|v|^{2}dy \\
&\leq C_{\lambda,K}\rho_0^{-4}e^{2\beta(\rho_0/2)^{-s}}\int_{\rho_0/2<|y|<2\rho_0/3}(|v|^{2}+|y|^{2}\left|\nabla v\right|^{2})dy\notag\\
&+C_{\lambda,K}\rho_2^{-4}e^{2\beta(\rho_2/2)^{-s}}\int_{\rho_2/2<|y|<2\rho_2/3}(|v|^{2}+|y|^{2}\left|\nabla v\right|^{2})dy.\notag
\end{align}
Adding $e^{2\beta \rho_{1}^{-s}}\rho_1^{-4}\int_{|y|<2\rho_{0}/3}|v|^{2}dy$ to both sides of \eqref{eq66} and using the Cacciopoli estimate, we obtain that
\begin{eqnarray*}
\int_{|y|<\rho_{1}}|v|^{2}dy \leq C_{\lambda,K}(\rho_1/\rho_0)^{4}e^{2\beta [(\rho_{0}/2)^{-s}-\rho_{1}^{-s}]}\int_{|y|<\rho_{0}}|v|^{2}dy\notag\\
+C_{\lambda,K}(\rho_2/\rho_1)^{4}e^{2\beta [(\rho_{2}/2)^{-s}- \rho_{1}^{-s}]}\int_{|y|<\rho_2}|v|^{2}dy,
\end{eqnarray*}
for all $\beta \ge s$. By standard arguments, we can choose a suitable choice for $\beta$ to deduce that
\[
\int_{|y|<\rho_1}|v|^{2}dy\le C\left(\int_{|y|<\rho_0}|v|^{2}dy\right)^{\tau}\left(\int_{|y|<\rho_2}|v|^{2}dy\right)^{1-\tau},
\]
where
\[
\tau=\frac{(2\rho_1)^{-s}-\rho_2^{-s}}{\rho_0^{-s}-\rho_2^{-s}}
\]
and the constant $C$ depends on $\lambda,K, s, \rho_1/\rho_0,\rho_2/\rho_0$. 

We note that the eigenvalues of $S$ are bounded by $\lambda ^{-1/2}$ and $\lambda ^{1/2}$, hence we have $B_{r_0}\supset S^{-1}B_{\rho_0}$, $B_{r_1}\subset S^{-1}B_{\rho_1}$ and $B_{r_2}\supset S^{-1}B_{\rho_2}$, where $B_{r_1}, B_{r_2}, B_{r_3}$ are balls in the $x$ variables.  Now reversing the change of variables leads to 
\[
\int_{B_{r_1}}|u|^{2}dx\le C\left(\int_{B_{r_0}}|u|^{2}dx\right)^{\tau}\left(\int_{B_{r_2}}|u|^{2}dx\right)^{1-\tau},
\]
where \[
\tau=\frac{(2r_1/\lambda)^{-s}-r_2^{-s}}{r_0^{-s}-r_2^{-s}}.
\]

\end{proof}

\section{Size estimate}

In this section, we study the size estimate problem described in the Introduction. As above, let us denote the background admittivity $\gamma_0=\sigma_0+i\epsilon_0$, where $\sigma_0$, $\epsilon_0$ are symmetric matrix valued functions satisfying for some $\alpha_0>0$
\begin{equation}\label{se0}
\alpha_0^{-1}I\le\sigma_0\le\alpha_0I,\quad \|\eps_0\|_{L^\infty(\Omega)}\le\alpha_0.
\end{equation}  
We further assume that $\gamma_0$ is Lipschitz, i.e.,
\begin{equation}\label{b0}
\|\nabla\gamma_0\|_{L^{\infty}(\Omega)}\le L.
\end{equation}
On the other hand, for the perturbed coefficients, we suppose that there exists $\alpha_1>0$ such that
\begin{equation}\label{se1}
(\sigma_1+\zeta_1)^{-1}\le\alpha_1I,\quad(\sigma_1-\zeta_1)\le\alpha_1I,\quad\|\eps_1\|_{L^\infty(\Omega)}\le\alpha_1.
\end{equation}
Denote $M=(\sigma_1+\zeta_1)^{-1}-\sigma_0^{-1}$, $N =(\sigma_1-\zeta_1)-\sigma_0$. We further assume the following jump conditions: there exist $\beta>0$ such that 
\begin{equation}\label{a0}
\beta I\leq M,\quad \beta I\leq N,\quad\forall\; x\in D\;\, \mbox{a.e.},
\end{equation}
and a constant $\delta=\delta(\alpha_0,\alpha_1,\beta)$ that will be determined later such that
\begin{equation}\label{delta}
\|\eps_1-\eps_0\|_{L^\infty(\Omega)}=\|\eps_1-\eps_0\|_{L^\infty(D)}\le\delta.
\end{equation}
From the fact that $\sigma_0$ is positive-definite and \eqref{a0}, it follows immediately that $(\sigma_1+\zeta_1)^{-1}$ and $(\sigma_1-\zeta_1)$ are positive-definite. 

We now can state the theorem of size estimate. Given a set $A$ and $s>0$, we defines  \[
A_\ell=\{x\in A:\text{dist}(x,\partial A)>\ell\}.
\]
\begin{thm}\label{size}
Assume that $\Omega \subset \mathbb{R}^n$ is open bounded and for some $\ell_0,\ell_1>0$, the inclusion  $D\subset\Omega$  satisfies $\text{dist}(D,\partial\Omega)\ge \ell_0$ and 
\begin{equation}\label{i1}
\quad|D_{\ell_1}|\ge\frac 12|D|.
\end{equation}
Also, let coefficients $\sigma_j, \ep_j, \zeta_j$, $j=0,1$, satisfy \eqref{se0}-\eqref{a0}. Then there exist two positive constants $K_1$, and $K_2$ such that
\[
K_1\frac{\Re\delta W}{\Re W_0}\le|D|\le K_2\frac{\Re\delta W}{\Re W_0},
\]
where $K_1$ depends on $\ell_0,\alpha_0, \alpha_1$ and $K_2$ depends on $\alpha_0$, $\alpha_1$, $ L$, $\beta$, $|\Omega|$, $\ell_1$, and $\|h\|_{L^{2}(\partial\Omega)}/\|h\|_{H^{-1/2}(\partial\Omega)}$.
\end{thm}

We will first derive an energy inequality which is crucial in the proof.
\begin{prop}\label{energy}
Assume that \eqref{se0}, \eqref{se1}, \eqref{a0} hold. Then there exist two positive constants $C_1$ and $C_2$ such that
\begin{equation}\label{eineq1}
C_1\int_D|\nabla u_0|^2\le \Re\delta W \le C_2\int_D|\nabla u_0|^2
\end{equation}
where $C_1$ depends on $\alpha_0,\alpha_1,\beta$ and $C_2$ depends on $\alpha_0,\alpha_1$.
\end{prop}
\begin{proof}
First we note that for $j=0,1$,
\begin{equation}\label{coneq1}
\begin{pmatrix} \Re\gamma_j(\nabla u_j) \\ \Im\gamma_j(\nabla u_j) \end{pmatrix} = \begin{pmatrix}\sigma_j +\zeta _j &-\eps_j\\ \eps_j&\sigma_j-\zeta _j\end{pmatrix} \begin{pmatrix}\Re \nabla u_j\\ \Im \nabla u_j\end{pmatrix}.
\end{equation}
We want to remind that we set $\zeta_0=0$. Following Cherkaev and Gibiansky \cite{cg94}, we rewrite \eqref{coneq1} as
\begin{equation}\label{coneq2}
\begin{pmatrix} \Re \nabla u_j\\ \Im\gamma_j(\nabla u_j) \end{pmatrix}=\begin{pmatrix}(\sigma_j+\zeta_j)^{-1}&(\sigma_j+\zeta_j)^{-1}\eps_j\\ \eps_j(\sigma_j+\zeta_j)^{-1}&\sigma_j-\zeta_j+\eps_j(\sigma_j+\zeta_j)^{-1}\eps_j\end{pmatrix}\begin{pmatrix}\Re\gamma_j(\nabla u_j) \\ \Im \nabla u_j\end{pmatrix}.
\end{equation}
We denote the square matrix on the right hand side $B_j$ and $v_j=(\Re\gamma_j(\nabla u_j),\Im \nabla u_j)^t$. We can easily check that $B_j$ are positive-definite. Then for $j,k\in\{0,1\}$, we have
\begin{equation}\label{basic}
\begin{aligned}
\int_\Omega B_j v_j\cdot v_k &=  \int_\Omega \Re\nabla u_j \cdot \Re\gamma_k(\nabla u_k) + \Im\gamma_j(\nabla u_j) \cdot \Im \nabla u_k\\
&=  \Re\int_\Omega \div(\gamma_k(\nabla u_k) \Re u_j) + \Im\int_\Omega \div(\gamma_j(\nabla u_j) \Im u_k)\\
&=\int _{\partial\Omega} ( \Re u_j \Re h + \Im u_k \Im h).
\end{aligned}
\end{equation}
Using \eqref{basic} and $B_0=B_0^T$, we obtain 
\begin{equation}
\begin{aligned}
\int_\Omega B_0 (v_0-v_1)\cdot(v_0-v_1)  =& \int_\Omega ( B_0 v_0\cdot v_0 -2 B_0 v_0\cdot v_1 + B_1 v_1\cdot v_1) \\
&+ \int_\Omega (B_0-B_1)v_1\cdot v_1\\
= &\quad \Re\int _{\partial\Omega} (g_1-g_0)h+\int_D (B_0-B_1)v_1\cdot v_1.
\end{aligned}
\end{equation}
Switching the indices 0 and 1, we obtain a similar identity. Thus,
\begin{equation}\label{id}
\begin{aligned}
\Re \delta W&= \int_\Omega B_0 (v_0-v_1)\cdot(v_0-v_1) + \int_D (B_1-B_0)v_1\cdot v_1 \\
&=- \int_\Omega B_1 (v_0-v_1)\cdot(v_0-v_1) + \int_D (B_1-B_0)v_0\cdot v_0, 
\end{aligned}
\end{equation}
where we used the fact that $\supp(B_1-B_0)\subset\bar D$. 
Recall that $B_1$ is positive-definite. The second equality of \eqref{id} gives the second inequality of \eqref{eineq1}.

To prove the first inequality of  \eqref{eineq1}, using the first equality of \eqref{id}, it suffices to show $B_1 -B_0$ is positive-definite. Then the inequality follows by the estimate of $B_0$ and the triangle inequality. To show that $B_1-B_0$ is positive-definite, we can write $B_1 -B_0 = C+D$, where
\[
C= \begin{pmatrix}M &M\eps_1\\ \eps_1 M& N+ \eps_1 M\eps_1\end{pmatrix}, \quad D= \begin{pmatrix}0 &\sigma_0^{-1}(\eps_1-\eps_0)\\(\eps_1-\eps_0)\sigma_0^{-1} &  (\eps_1-\eps_0)\sigma_0^{-1}  \eps_1+ \eps_0 \sigma_0^{-1} (\eps_1-\eps_0)\end{pmatrix}.
\]
For $X=(u,v)^T\in \R^{2n}$ We have 
\[
CX\cdot X=M (u+\eps_1 v)\cdot (u+\eps_1 v)
+ Nv\cdot v
\]
and 
\begin{align*}
\begin{split}
DX\cdot X &=(2u+(\eps_0+\eps_1)v)\cdot \sigma_0^{-1}(\eps_1-\eps_0) v\\
&=2u\cdot\sigma_0^{-1}(\eps_1-\eps_0)v+(\eps_1-\eps_0)v\cdot\sigma_0^{-1}(\eps_1-\eps_0)v+\eps_0 v\cdot\sigma_0^{-1}(\eps_1-\eps_0)v.
\end{split}
\end{align*}
The required positivity then follows from condition \eqref{a0} and \eqref{delta} with a small $\delta$ depending on $\alpha_0,\alpha_1,\beta$.
\end{proof}

Another tool we need is the following propagation of smallness which is a consequence of the Carleman estimate.
\begin{prop}\label{smallness}
Let $u_0$ be the solution of \eqref{unpert}. Assume that \eqref{se0} and \eqref{b0} hold. Then for any $\rho>0$ and every
$x\in\Omega_{4\rho}$, we have
\begin{equation}
\int_{B_{\rho}(x)}|\nabla u_0|^{2}\ge C_{\rho}\int_{\Omega}|\nabla u_0|^{2},\label{lp}
\end{equation}
where $C_{\rho}$ depends on $\alpha_0$, $L$, $|\Omega|$, $\rho$, $\|h\|_{L^{2}(\partial\Omega)}/\|h\|_{H^{-1/2}(\partial\Omega)}$.
\end{prop}
\begin{proof}

We follow the arguments presented in \cite[Lemma~2.2]{ars}.
We first observe that it suffices to consider the case $\rho$ is
small, so we can assume that $\Omega_{\rho}$ is connected. Using
Caccioppoli and Poincar\'e inequalities, we can deduce from Theorem~\ref{T1}
that
\begin{equation}
\|\nabla u_0\|_{L^{2}(B_{3r}(x))}\le C\|\nabla u_0\|_{L^{2}(B_{r}(x))}^{\tau}\|\nabla u_0\|_{L^{2}(B_{4\lambda^{-1}r}(x))}^{1-\tau},\label{3b2}
\end{equation}
where $C$ depends on $\alpha_0,\beta_0$. Given $x,y\in\Omega_{4\lambda^{-1}\rho}$, let $\gamma$ be a curve in $\Omega_{4\lambda^{-1}\rho}$
joining $x$ and $y$. We define a sequence $x_{k}$'s as follows:
Let $x_{1}=x$. For $k>1$, let $x_{k}=\gamma(t_{k})$ where $t_{k}=\max\{t:|\gamma(t)-x_{k-1}|=2\rho\}$
if $|x_{k}-y|>2\rho$; otherwise let $x_{k}=y$, $N=k$ and stop the
process. Note that since the balls $B_{\rho}(x_{k})$ are disjoint,
$N\le N_{0}=\frac{|\Omega|}{\omega_n\rho^{n}}$. Using \eqref{3b2} with $x=x_k$ and $r=\rho$, noting
that $B_{\rho}(x_{k+1})\subset B_{3\rho}(x_{k})$ because $|x_{k+1}-x_{k}|\leq2\rho$,
we can deduce that
\[
\frac{\|\nabla u_0\|_{L^{2}(B_{\rho}(x_{k+1}))}}{\|\nabla u_0\|_{L^{2}(\Omega)}}\le C\left(\frac{\|\nabla u_0\|_{L^{2}(B_{\rho}(x_{k}))}}{\|\nabla u_0\|_{L^{2}(\Omega)}}\right)^{\tau}.
\]
Here $C$ depends on $\alpha_0,\beta_0,\rho$. By induction, we obtain
\begin{equation*}
\frac{\|\nabla u_0\|_{L^{2}(B_{\rho}(y))}}{\|\nabla u_0\|_{L^{2}(\Omega)}}\le C^{1/(1-\tau)}\left(\frac{\|\nabla u_0\|_{L^{2}(B_{\rho}(x))}}{\|\nabla u_0\|_{L^{2}(\Omega)}}\right)^{\tau^{N}}.
\end{equation*}
Since we can cover $\Omega_{(4\lambda^{-1}+1)\rho}$ by no more than $\frac{|\Omega|n^{n/2}}{2^n\rho^{n}}$
balls of radius $\rho$, we obtain
\begin{equation*}
\frac{\|\nabla u_0\|_{L^{2}(\Omega_{(4\lambda^{-1}+1)\rho})}}{\|\nabla u_0\|_{L^{2}(\Omega)}}\le C\left(\frac{\|\nabla u_0\|_{L^{2}(B_{\rho}(x))}}{\|\nabla u_0\|_{L^{2}(\Omega)}}\right)^{\tau^{N_{0}}},
\end{equation*}
where $C$ depends on $\alpha_0$, $\beta_0$, $|\Omega|$, and $\rho$. 

Finally, repeating arguments in \cite[page 60-61]{ars}, we have that there exists $\bar\rho$, depending on $\alpha_0,\beta_0,|\Omega|, \|h\|_{L^2(\partial\Omega)}/\|h\|_{H^{-1/2}(\partial\Omega)}$, such that
\begin{equation}\label{e5}
\frac 12\le\frac{\|\nabla u_0\|_{L^{2}(\Omega_{(4\lambda^{-1}+1)\rho})}}{\|\nabla u_0\|_{L^{2}(\Omega)}}
\end{equation}
for all $0<\rho\le\bar\rho$. Thus the estimate \eqref{lp} holds for all $0<\rho\le\bar\rho$ and for large $\rho$, \eqref{lp} is then obvious. 
\end{proof}

Finally, we give the proof of our main result. 
\begin{proof}[Proof of Theorem \ref{size}]
We follow  the proof of Theorem 2.1 in \cite[page 61]{ars}. By the standard elliptic estimate, we have the interior estimate
$$
\sup_{D}|\nabla u_0|\leq C||\nabla u_0||_{L^2(\Omega)}\leq C \left(\Re W_0\right)^{1/2}.
 $$ 
 Applying the second inequality of Proposition \ref{energy} we obtain
$$
 \Im\delta W \leq C_2|D|\left(\sup_D |\nabla u_0|\right)^2\leq K_1^{-1} |D|\Re W_0, 
$$
which is the first of our needed inequalities.

Now we let $d=\min\left(d_0/2,d_1/2\right)$ and cover $D_{d_1}$ with squares of side $d$ having disjoint interiors, $Q_k\subset D$, $k=1,\ldots,N$. It is clear that $N\geq d^{-2} |D_{d_1}|\ge\frac{1}{2d^2}|D|$. Applying Proposition \ref{smallness} with $\rho=d/2$ we see that $\int_{Q_k}|\nabla u_0|^2\geq C\Re W_0$, hence

$$
\int_D |\nabla u_0|^2\geq \sum_{k=1}^N\int_{Q_k}|\nabla u_0|^2\geq N C\Re W_0 \geq C|D| \Re W_0.
$$
Combining this with  the first inequality of Proposition \ref{energy}, the proof is finished.
\end{proof}

We note that  \eqref{i1} is only used to obtain an upper bound for $|D|$. In \cite{bfv}, an estimate  was obtained without assuming \eqref{i1}. This was possible because under their conditions, $u_0$ satisfies a  doubling inequality, which is not known in the general case considered here.

\end{document}